\documentclass[reqno,11pt]{amsart}

\usepackage{euscript}
\usepackage{enumerate}
\usepackage{amsmath}
\usepackage{mathrsfs}
\usepackage{mathtools}
\usepackage{enumitem}
\usepackage[initials]{amsrefs}
\usepackage{amssymb}
\usepackage{hyperref}

\theoremstyle{plain}
\newtheorem{thm}{Theorem}[section]
\newtheorem{prop}[thm]{Proposition}
\newtheorem{cor}[thm]{Corollary}
\newtheorem{ques}[thm]{Question}
\newtheorem{lem}[thm]{Lemma}

\theoremstyle{remark}
\newtheorem{rem}[thm]{Remark}

\theoremstyle{definition}
\newtheorem{defn}[thm]{Definition}
\newtheorem{ex}[thm]{Example}

\newcommand\CC{{\mathbf C}}
\newcommand\NN{{\mathbf N}}
\newcommand\NC{\operatorname{NC}}
\newcommand\TT{{\mathbf T}}
\newcommand\ZZ{{\mathbf Z}}

\newcommand{\bub}[1]{\mathring{#1}}
\newcommand{\myref}[2]{\hyperref[#2]{#1~\ref*{#2}}}
\newcommand{\myrefsplit}[2]{\hyperref[#2]{#1 \ref*{#2}}}

\newcommand{\AND}{\text{and}}
\newcommand{\Alg}{\operatorname{Alg}}
\newcommand\At{{\widetilde A}}
\newcommand\Et{{\widetilde E}}

\newcommand{\lspan}{\operatorname{span}}
\newcommand{\OR}{\text{or}}

\newcommand\ut{{\tilde u}}
\newcommand\xt{{\tilde x}}

\begin{document}

\title[Haar unitaries and bipolar decompositions]{On operator valued Haar unitaries and bipolar decompositions of R-diagonal elements}

\author[Dykema]{Ken Dykema}
 \address{K. Dykema, Department of Mathematics, Texas A\&M University, College Station, TX 77843-3368, USA.}
 \email{kdykema@math.tamu.edu}
%

\author[Griffin]{John Griffin}
\address{J. Griffin}
\thanks{Portions of this work are included in the dissertation of J.\ Griffin for partial fulfillment of the requirements to obtain a Ph.D.\ degree
at Texas A\&M University.}
\email{jagriffin3rd@gmail.com}



\begin{abstract} 
In the context of operator valued W$^*$-free probability theory, we study Haar unitaries, R-diagonal elements and circular elements.
Several classes of Haar unitaries are differentiated from each other.
The term bipolar decomposition is used for the expression of 
an element as $vx$ where $x$ is self-adjoint and $v$ is a partial isometry,
and we study such decompositions of operator valued R-diagonal and circular elements that are free, meaning that $v$ and $x$ are $*$-free from each other.
In particular, we prove, when $B=\CC^2$, that if a $B$-valued circular element has a free bipolar decomposition with $v$ unitary,
then it has one where $v$ normalizes $B$.
\end{abstract}

\date{17 January, 2024}

\maketitle

\section{Introduction}

  Consider a tracial W$^*$-noncommutative probability space $(A,\tau)$, namely a von Neumann algebra $A$, with a normal faithful tracial state $\tau$.
Voiculescu's circular operator, introduced in~\cite{V91}, arises naturally in free probability theory.
Voiculescu proved~\cite{V-paper} that a circular operator has polar decomposition $z=u|z|$, where $u$ is a Haar unitary, namely, $\tau(u^n)=0$
for all integers $n\ge1$ and where $u$ and $|z|$ are $*$-free.

R-diagonal elements were introduced by Nica and Speicher~\cite{NS-paper}.
They are defined by conditions involving free cumulants, but it turns out that, in the tracial setting, they are the operators of the form $uh$
where $u$ is a Haar unitary, $h\ge0$ and $u$ and $h$ are free;
for example, the circular operator and the Haar unitary are R-diagonal elements.

Operator-valued free probability theory, introduced by Voiculescu~\cite{V85}, is (in this context of W$^*$-algebras) the generalizeation wherein
the scalar field $\CC\subseteq A$ is replaced by a W$^*$-subalgebra $B\subseteq A$ and where $\tau$ is replaced by a normal conditional expectation $E:A\to B$.
The elements of $A$ are called $B$-valued random variables.
We will frequently be interested in such pairs $(A,E)$ when for some normal, faithful tracial state $\tau_B$ on $B$,
the composition $\tau_B\circ E$ is a trace on $A$.
We will say that such $B$-valued W$^*$-noncommutative probability spaces $(A,E)$ are tracial, and refer to random variables $a\in A$ as tracial.
$B$-valued R-diagonal elements, defined in terms of $B$-valued cumulants, were introduced and studied by {\'S}niady and Speicher in~\cite{SS-paper}
and have also been studied in~\cite{BD-paper}.
(See \myref{Section}{sec:background} for more details.)

We would like to understand, for $B$-valued R-diagonal elements $a$ that are tracial, more about the polar decomposition $a=u|a|$ of $a$,
in the case when $u$ is unitary.
Which sort of unitaries $u$ can appear?
Must $u$ and $|a|$ be $*$-free over $B$?  (We would then say that $a$ has a free polar decomposition.)
We know (by Example~6.8 of~\cite{BD-paper}) that the second question has a negative answer.

In \myref{Section}{section:unitaries}, we consider several classes of Haar unitaries and give examples to differentiate these classes.
As it turns out, the unitaries $u$ that can appear as the polar parts of $B$-valued R-diagonal elements are precisely the $B$-valued R-diagonal unitaries
(see \myref{Proposition}{prop:polar are r-diagonal}).
One example of a $B$-valued R-diagonal unitary is a Haar unitary $u$ that normalizes $B$, in the sense that for all $b\in B$ we have $ubu^*=\theta(b)$,
for some automorphism $\theta$ of $B$.
However, there are others (see \myref{Example}{ex:RdiagNotNormalizing}).

We are at present unable to find natural criteria for a $B$-valued R-diagonal element $a\in A$ to have a free polar decomposition.
However, we introduce (in \myref{Section}{sec:bipolar}) a related and slightly looser notion of a {\em bipolar decomposition}, for which such criteria exist. 
A bipolar decomposition for a $B$-valued random variable $a\in A$ is a pair $(v,x)$ (possibly in some other $B$-valued W$^*$-noncommutative probability space)
so that $a$ and $vx$ have the same $*$-moments, where $x$ is self-adjoint, $v$ is a partial isometry and $v^*v$ is greater than or equal to the support projection of $x$.
We are interested in bipolar decompositions that are {\em free} in the sense that $v$ and $x$ are $*$-free, and where $v$ is unitary.
In~\cite{BD-paper}, some natural conditions in terms of cumulants (or of moments) of a $B$-valued R-diagonal element $a\in A$ were derived, that are equivalent to
$a$ having a free bipolar decomposition $(u,x)$ with $u$ a unitary that normalizes $B$.
It is an interesting question: what other sorts of unitaries $u$ can appear in free bipolar decompositions of $B$-valued R-diagonal random variables?

The $B$-valued circular elements are very special cases of $B$-valued R-diagonal elements, namely, where the only nonvanishing cumulants are of order two.
(See \myref{Section}{sec:background} for details.)
It is not the case that every tracial $B$-valued circular element has a free bipolar decomposition.
Example~6.8 of~\cite{BD-paper} provides a counter-example when $B=\CC^2$.
In an attempt to understand more about unitaries that can arise in free bipolar decompositions of R-diagonal elements,
and hoping to find some more exotic examples of such unitaries,
we focused on the case of $B$-valued circular variables where $B=\CC^2$.
In \myref{Section}{section:c2}, we investigate the tracial $\CC^2$-valued circular elements that do have free bipolar decompositions of the form $(u,x)$ with $u$ unitary,
and we show that all of them do have free bipolar decompositions of this form where $u$ normalizes $B$.
The proof depends on calculations involving polynomials that were performed in Mathematica~\cite{Mat}, and we make available the Mathematica Notebook file in which
these calculations were performed~\cite{DG}.
While we don't obtain any exotic unitaries, it raises the question of whether this situation holds for $B$-valued circular elements for all $B$, or whether it is an artifact of the
low dimension of $\CC^2$.
This question and one other open question  are found in \myref{Section}{sec:qns}.

\section{Background and Notation}
\label{sec:background}

In this section, we recall some definitions and results that we need and fix notation we will use.

\begin{defn}
Let $B$ be a unital $*$-algebra over the complex numbers.
A \textit{$B$-valued $*$-noncommutative probability space} is a pair
$(A,E)$, where $A$ is a unital $*$-algebra containing a unitally embedded copy
of $B$ and $E:A\to B$ is a positive, $*$-preserving, idempotent linear function that restricts
to the identity on $B$ and satisfies $E(b_1ab_2)=b_1E(a)b_2$ for every
$a\in A$ and $b_1,b_2\in B$.
If $B$ and $A$ are both C$^*$-algebras,
then we say $(A,E)$ is a
\textit{$B$-valued C$^*$-noncommutative probability space}, and $E$ is automatically of norm $1$.
If $B$ and $A$ are both von Neuamm algebras and $E$ is normal, 
then we call $(A,E)$ a
\textit{$B$-valued W$^*$-noncommutative probability space}.

The mapping $E$ is called a \textit{conditional expectation}.
Note that positivity of $E$ 
Elements of $A$ are called \textit{$B$-valued random variables}, or simply
\textit{random variables}.
\end{defn}

\begin{defn}
Let $a\in A$ be a $B$-valued random variable in the $B$-valued $W^*$-noncommutative probability space $(A,E)$.
We say that $a$ (or the $*$-distribution of $a$) is {\em tracial} if there exists a normal, faithful, tracial state $\tau_B$ on $B$ such that
the restriction of $\tau_B\circ E$ to the $*$-algebra generated by $a$ and $B$ is a trace.
\end{defn}

Henceforth we fix a $B$-valued $*$-noncommutative probability space $(A,E)$
over some unital $*$-algebra $B$.

The $B$-valued probabilistic information about a random variable is captured by
its $B$-valued $*$-moments and $B$-valued $*$-distribution.

\begin{defn}
A \textit{$B$-valued moment} of a family $(a_i)_{i\in I}$ of random variables is
an expression of the form
\[
   E(a_{\varepsilon(1)}b_1\cdots a_{\varepsilon(n-1)}b_{n-1}a_{\varepsilon(n)}),
\]
where $n\in\NN$, $b_1,\ldots,b_{n-1}\in B$, and $\varepsilon\in I^n$.
We typically consider a \textit{$B$-valued $*$-moment} of a single random
variable $a$, which is a $B$-valued moment of the pair $(a,a^*)$.

Let $B\langle (X_i)_{i\in I}\rangle$ be the universal algebra over $\CC$
generated by $B$ and the indeterminates $(X_i)_{i\in I}$.
The \textit{$B$-valued distribution} of a family $(a_i)_{i\in I}$
is the mapping $\Theta:B\langle (X_i)_{i\in I}\rangle \to B$
defined by extending
\[
   \Theta(b_0X_{\varepsilon(1)}b_1\cdots X_{\varepsilon(n)}b_n) 
   = E(b_0a_{\varepsilon(1)}b_1\cdots a_{\varepsilon(n)}b_n)
\]
linearly.
When we refer to the \textit{$B$-valued $*$-distribution} of a single random
variable $a$, we mean the $B$-valued distribution of the pair $(a,a^*)$.
\end{defn}

Here is Voiculescu's notion of free indpendence in the $B$-valued setting.
\begin{defn}
Let $(A_i)_{i\in I}$ be a family of $*$-subalgebras of $A$, such that each $A_i$
containes the copy of $B$ that is embedded in $A$.
This family of subalgebras is \textit{free over $B$} if
$E(a_1\cdots a_n) = 0$
whenever:
\begin{enumerate}[label=(\alph*)]
\item $n$ is a positive integer;
\item $\varepsilon\in I^n$ so that $a_i\in A_{\varepsilon(i)}$ for each
      $1\le i \le n$;
\item Each $a_i$ is centered, namely, $E(a_i)=0$ for all $1 \le i \le n$;
\item Neighboring elements are from different subalgebras, meaning, 
      $\varepsilon(1)\ne\varepsilon(2),\varepsilon(2)\ne\varepsilon(3),
         \ldots,\varepsilon(n-1)\ne\varepsilon(n)$.
\end{enumerate}
A family of subsets $(X_i)_{i\in I}$ is \textit{$*$-free over $B$} if
the corresponding family of $*$-subalgebras $($*$-\Alg(B\cup X_i))_{i\in I}$
is $*$-free.
A family $(a_i)_{i\in I}$ of random variables $a_i\in A$ is {\em free over $B$} or, respectively, {\em $*$-free over $B$}
if the family $(\{a_i\})_{i\in I}$ is free over $B$ or, respectively, $*$-free over $B$.
\end{defn}

Of course, the definition of $*$-freeness depends on the conditional
expectation $E$.
If there is any ambiguity in which conditional expectation our discussion of
$*$-freeness over $B$ concerns, we'll replace ``over $B$"
with the terminology
``with respect to E".
If $B=\CC$, then $*$-freeness over $B$ is the same as $*$-freeness in the
scalar setting.

We present a formulation from \cite{NSS-paper} of
the $B$-valued cumulant maps
that were introduced by Speicher~\cite{S-paper}.
\begin{defn}
Suppose $(a_i)_{i\in I}$ is a family of $B$-valued random variables.
Write $J=\bigcup_{n\in\NN}I^n$.
Given $n\in\NN$ and $j\in I^n$, the corresponding cumulant
map $\alpha_j:B^{n-1}\to B$ is defined recursively by the
\textit{$B$-valued moment-cumulant formula}
\begin{equation}
   E(a_{j(1)}b_1a_{j(2)}\cdots b_{n-1}a_{j(n)})
   = \sum_{\pi\in\NC(n)}\hat{\alpha}_j(\pi)[b_1,\ldots,b_{n-1}],
\end{equation}
where $\NC(n)$ is the set of all noncrossing partitions of $\{1,\ldots,n\}$
and, given $\pi\in\NC(n)$, $\hat{\alpha}_j(\pi)$ is a multilinear map defined
in terms of $\alpha_{j'}$ for each $j'$ obtained by restricting $j$ to a
block of $\pi$.
More precisely, if $\pi=1_n$, then
\[
   \hat{\alpha}_j(\pi)[b_1,\ldots,b_{n-1}]=\alpha_j(b_1,\ldots,b_{n-1}),
\]
and for $\pi\ne 1_n$, we choose an interval block $\{p,p+1,\ldots,p+q-1\}\in\pi$
with $p\ge 1$ and $q\ge 1$, and let $\pi'\in\NC(n-q)$ be obtained by restricting
$\pi$ to $\{1,\ldots,p-1\}\cup\{p+q,\ldots,n\}$ and then renumbering to
preserve order.
Then, with $j'=(j(1),\ldots,j(p-1),j(p+q),\ldots,j(n))\in I^{n-q}$ and
$j''=(j(p),\ldots,j(p+q-1))\in I^q$, we have
\begin{align*}
   \hat{\alpha}_j&(\pi)[b_1,\ldots,b_{n-1}] \\
   &= \begin{cases}
      \hat{\alpha}_{j'}(\pi')[b_1,\ldots,b_{p-2}, \\
         \qquad\qquad b_{p-1}\alpha_{j''}(b_p,\ldots,b_{p+q-2})b_{p+q-1}, \\
         \qquad\qquad b_{p+q},\ldots,b_{n-1}],
         & p\ge 2,\, p+q-1<n, \\
      \hat{\alpha}_{j'}(\pi')[b_1,\ldots,b_{p-2}]b_{p-1}
         \alpha_{j''}(b_p,\ldots,b_{n-1}),
         & p\ge 2,\, p+q-1 = n, \\
      \alpha_{j''}(b_1,\ldots,b_{q-1})b_q
         \hat{\alpha}_{j'}(\pi')[b_{q+1},\ldots,b_{n-1}],
         & p=1,\, q<n.
   \end{cases}
\end{align*}

We call the collection of maps $(\alpha_j)_{j\in J}$ the
\textit{$B$-valued cumulants} of $(a_i)_{i\in I}$.
The \textit{$B$-valued $*$-cumulants} of a random variable $a$ are the
$B$-valued cumulants corresponding to the pair $(a,a^*)$.
\end{defn}

R-diagonal elements were introduced in the scalar-valued case in~\cite{NS-paper}
and in the more general $B$-valued case in~\cite{SS-paper}.
The equivalent definition below is found in~\cite{BD-paper}, and was inspired by results of~\cite{NSS01}.
Let $B$ be a unital $*$-algebra and $(A,E)$ be a $B$-valued $*$-noncommutative
probability space.

\begin{defn}
Given $n\in\NN$ and $\varepsilon\in\{1,*\}^n$, the \textit{maximal alternating
interval partition} $\sigma(\varepsilon)$ is the interval partition of
$\{1,\ldots,n\}$ whose blocks $V$ are the maximal interval subsets of
$\{1,\ldots,n\}$ such that if $j\in V$ and $j+1\in V$, then
$\varepsilon(j)\ne\varepsilon(j+1)$.

\label{defn:R-diagonal}
A $B$-valued random variable $a$ is said to be \textit{$B$-valued R-diagonal}
if for every $k\ge 0$ and $b_1,\ldots,b_{2k}\in B$ we have
\[
   E(ab_1a^*b_2ab_3a^*\cdots b_{2k-2}ab_{2k-1}a^*b_{2k}a) = 0,
\]
(i.e., odd alternating moments vanish) and, for $n\ge 1$,
$\varepsilon\in\{1,*\}^n$, and $b_1,b_2,\ldots,b_n\in B$ we have
\[
   E\left(\prod_{V\in\sigma(\varepsilon)}
      \left(\left(\prod_{j\in V}a^{\varepsilon(j)}b_j\right)
         -E\left(\prod_{j\in V}a^{\varepsilon(j)}b_j\right)
      \right)\right) = 0,
\]
where the terms in each of the three products are taken in the order of
increasing indices.
\end{defn}

Before listing some characterizations of R-diagonal random variables, 
we'll need some more notation.

\begin{defn}\label{def:HaarU}
A unitary $u\in A$ is called a \textit{Haar unitary} if $E(u^k)=0$ for all  nonzero integers
$k$.
\end{defn}

The name arises from the case $B=\CC$, when the condition is equivalent to the distribution measure of $u$ being Haar measure on the circle.
In fact, when $B=\CC$, this condition is enough to determine the
$*$-distribution of $u$, and implies that $u$ is R-diagonal.
However, in the $B$-valued setting, a Haar unitary need not be R-diagonal.
This is explored further in \myref{Section}{section:unitaries}.

\begin{defn}
A $B$-valued $*$-noncommutative probability space $(A',E')$ is said to be
an \nobreak \textit{enlargement} of $(A,E)$ if there is an embedding
$\theta:A\to A'$ so that $\theta(b)=b$ for each $b\in B$ and
$E'(\theta(a))=E(a)$ for each $a\in A$.
\end{defn}

The following theorem, which appears in \cite{BD-paper} and parts of which
appear in \cite{SS-paper}, gives some useful characterizations of $B$-valued
R-diagonal elements.

\begin{thm}[\cite{BD-paper}, Theorem 3.1]
\label{thm:R-diagonal}
Let $a\in A$.
The following are equivalent:
\begin{enumerate}[label=(\alph*)]
\item $a$ is $B$-valued R-diagonal.
\item \label{part:R-diagonal-cumulant} 
      The only non-vanishing $B$-valued $*$-cumulants of $a$
      (namely, cumulants of the pair $(a_1,a_2)=(a,a^*)$) are those that are
      alternating and of even length.
      That is, $\alpha_j=0$ unless $j$ is of the form $(1,2,\ldots,1,2)$ or
      $(2,1,\ldots,2,1)$.
\item \label{part:R-diagonal-condition}
      There is an enlargement $(A',E')$ of $(A,E)$ and a unitary $u\in A'$
      such that
      \begin{enumerate}[label=(\roman*)]
      \item $u$ commutes with $B$,
      \item $u$ is a Haar unitary,
      \item $u$ and $a$ are $*$-free with respect to $E'$, and
      \item $a$ and $ua$ have the same $B$-valued $*$-distribution.
      \end{enumerate}
\item \label{part:R-diagonal-implication}
      If $(A',E')$ is an enlargement of $(A,E)$ and $u\in A'$ is a unitary
      such that
      \begin{enumerate}[label=(\roman*)]
      \item $u$ commutes with $B$ and
      \item $u$ and $a$ are $*$-free with respect to $E'$,
      \end{enumerate}
      then $a$ and $ua$ have the same $B$-valued $*$-distribution.
\end{enumerate}
\end{thm}

\myref{Condition}{part:R-diagonal-cumulant} of the above theorem will be
referred to as the \textit{cumulant property of R-diagonal elements},
whereas \myref{Definition}{defn:R-diagonal} will be called the
\textit{moment property of R-diagonal elements}.

In light of the cumulant property for a $B$-valued R-diagonal $a$, we can
simplify the notation for the $B$-valued $*$-cumulants of $a$.
Instead of considering the cumulants $(\alpha_j)_{j\in J}$, indexed by
$J=\bigcup_{n\in\NN}\{1,2\}^n$, we will work with the two sequences
$(\beta_k^{(1)})_k$ and
$(\beta_k^{(2)})_k$ defined by $\beta_k^{(1)} = \alpha_{(1,2,\ldots,1,2)}$
and $\beta_k^{(2)} = \alpha_{(2,1,\ldots,2,1)}$.

The cumulant property can also be used to define a special type of $B$-valued
R-diagonal element, namely, the circular elements.

\begin{defn}
A $B$-valued random variable $a\in A$ is \textit{$B$-valued circular} if it is $B$-valued R-diagonal and, in the notation introduced above for its cumulant maps $\alpha_j$
for $j\in J=\bigcup_{n\ge 1}\{1,2\}^n$,
we have
$\alpha_j=0$ whenever $j\notin\{(1,2),(2,1)\}$.
\end{defn}

\section{Classes of $B$-valued Haar unitaries}
\label{section:unitaries}

Let $B$ be a unital $*$-algebra and $(A,E)$ be a $B$-valued $*$-noncommutative
probability space.
Recall (\myref{Definition}{def:HaarU}) that a {\em Haar unitary} is a unitary element $u\in A$ such that $E(u^n)=0$ for every integer $n\ge1$.

\begin{defn}
Let $a\in A$ and consider a $*$-moment
\[
   E(a^{\varepsilon(1)}b_1a^{\varepsilon(2)}b_2\cdots a^{\varepsilon(n-1)}
      b_{n-1}a^{\varepsilon(n)}),
\]
with $n\in\NN$, $b_1,\ldots,b_{n-1}\in B$, and $\varepsilon\in\{1,*\}^n$.
We say that this $*$-moment is
\textit{unbalanced} if the number of $*$'s differ from the number
of non-$*$'s; i.e., if 
\[
   \#\{j\mid \varepsilon(j) = 1\} \ne \frac{n}{2}.
\]
An element $a\in A$ is called \textit{balanced} if all of its unbalanced
$*$-moments vanish.
An element that is not balanced is called \textit{unbalanced}.
\end{defn}

It's an easy consequence of the moment-cumulant formula and the cumulant
property of R-diagonal elements
(see \myref{Theorem}{thm:R-diagonal}\ref{part:R-diagonal-cumulant}) that every
$B$-valued R-diagonal element is balanced.
It's also clear that if a unitary is balanced, then it is a
\textit{Haar unitary};
i.e., $E(u^k)=0$ for all integers $k\ge 1$.

Recall that a unitary $u$ normalizes $B$ if $ubu^*,u^*bu\in B$ for every
$b\in B$.
It follows from the moment property of R-diagonal elements
(see \myref{Definition}{defn:R-diagonal})
that every normalizing Haar unitary is R-diagonal.
Indeed, if $u$ normalizes $B$, then for each $b\in B$ there exist $b',b''\in B$
satisfying $ub=b'u$ and $u^*b=b''u^*$.
Consequently, odd alternating moments vanish because, given any $k\ge 1$ and
$b_1,\ldots,b_{2k}\in B$, there is $b_0\in B$ such that
$ub_1u^*b_2\cdots b_{2k-2}ub_{2k-1}u^*b_{2k}u=b_0u$.
Fix $n\ge 1$, $\varepsilon\in\{1,*\}^n$, and $b_1,\ldots,b_n\in B$, and consider
the expression
\begin{equation}
\label{eqn:normalizing-r-diag}
   E\left(\prod_{V\in\sigma(\varepsilon)}\left(
      \left(\prod_{j\in V}u^{\varepsilon(j)}b_j\right)
      -
      E\left(\prod_{j\in V}u^{\varepsilon(j)}b_j\right)
      \right)
      \right).
\end{equation}
If $V_0$ is an alternating interval block in $\sigma(\varepsilon)$ of even
length, then the product $\prod_{j\in V_0}u^{\varepsilon(j)}b_j$ is in $B$,
and consequently the term
\[ \left(\prod_{j\in V_0}u^{\varepsilon(j)}b_j\right)
   -E\left(\prod_{j\in V_0}u^{\varepsilon(j)}b_j\right) \]
vanishes.
Thus, the expression~\eqref{eqn:normalizing-r-diag} vanishes whenever some block has even
length.
If there is no block of even length, then every block in $\sigma(\varepsilon)$
has odd length.
Thus
\[
   E\left(\prod_{j\in V}a^{\varepsilon(j)}b_j\right) = 0
\]
for each $V\in\sigma(\varepsilon)$, so expression~\eqref{eqn:normalizing-r-diag} becomes
\begin{align*}
   E\left(\prod_{V\in\sigma(\varepsilon)}\left(
      \left(\prod_{j\in V}u^{\varepsilon(j)}b_j\right)
      -
      E\left(\prod_{j\in V}u^{\varepsilon(j)}b_j\right)
      \right)
      \right)
   &= E(u^{\varepsilon(1)}b_1\cdots u^{\varepsilon(n)}b_n) \\
   &= b_0E(u^k)
\end{align*}
for some $b_0\in B$ and $k\in\ZZ$.
Since there are no even length blocks in the maximal alternating interval
partition, there must an unequal number of $1$'s and $*$'s in the list $\epsilon(1),\ldots,\epsilon(n)$.
Thus $k\ne0$, so~\eqref{eqn:normalizing-r-diag} vanishes by the Haar property.
This completes the proof that $u$ is R-diagonal.

The following is a summary of these basic facts for classes of unitary elements.
\begin{thm}\label{thm:unitaries}
Let $B$ be a $*$-algebra.
For a unitary $u$ in a $B$-valued $*$-non\-com\-mu\-ta\-tive probability space, consider the conditions
\begin{enumerate}
\item[(a)] $u$ is a normalizing Haar unitary;
\item[(b)] $u$ is an R-diagonal unitary;
\item[(c)] $u$ is a balanced unitary;
\item[(d)] $u$ is a Haar unitary.
\end{enumerate}
Then the implications (a)$\implies$(b)$\implies$(c)$\implies$(d) hold.
\end{thm}
In the case $B=\CC$, the notion of normalizing is redundant, because every unitary normalizes the
scalars; so the four conditions are equivalent in this case.
The next three examples will show that none of the reverse implications hold in \myref{Theorem}{thm:unitaries}, when considering general $B$.

\begin{ex}[Unbalanced Haar unitary]
\label{ex:unbalanced haar}
Let $A=M_2(C(\TT))$ be the algebra of $2\times 2$ matrices
with entries from the continuous functions on the circle.
We identify $B=\CC^2$ with the subalgebra of $A$ consisting of diagonal matrices
whose entries are constant functions.
Let $E:A\to B$ be the conditional expectation given by
\[
   E\left( \begin{pmatrix}f_{11} & f_{12}\\ f_{21} & f_{22}\end{pmatrix} \right)
   = \begin{pmatrix} \tau(f_{11}) & 0 \\ 0 & \tau(f_{22}) \end{pmatrix},
\]
where $\tau$ is the trace on $C(\TT)$ given by integration with respect to the
Haar measure on $\TT$.
This gives the $B$-valued C$^*$-noncommutative probability space $(A,E)$.

Let $v\in C(\TT)$ be the identity function; namely, $z\mapsto z$.
Then $v$ is a Haar unitary in $C(\TT)$ with respect to $\tau$.
By identifying $A$ with $M_2(\CC)\otimes C(\TT)$ in the usual way, we define
the unitary
\[ u=p\otimes v + (1-p)\otimes v^* \in A, \]
where the projection $p\in M_2(\CC)$ is given by
\begin{equation}
\label{eq:p2x2}
   p = \frac{1}{2}\begin{pmatrix} 1 & 1 \\ 1 & 1 \end{pmatrix}.
\end{equation}
Since $u^k = p\otimes v^k + (1-p) \otimes (v^*)^k$, we have
\[
   E(u^k) = \frac{1}{2}\begin{pmatrix}\tau(v^k)+\tau((v^*)^k) & 0
                           \\ 0 & \tau(v^k)+\tau((v^*)^k) \end{pmatrix}.
\]
Hence, $u$ is a Haar unitary.

To see that $u$ is unbalanced, we consider the matrix unit $e_{11}\in M_2(\CC)$ as an element of
$B$ and compute
\begin{align*}
   E(ue_{11}u)
   &= E\left(
      \frac{1}{4}\begin{pmatrix} v+v^* & 0 \\ v-v^* & 0 \end{pmatrix}
      \begin{pmatrix} v+v^* & v-v^* \\ 0 & 0 \end{pmatrix}
      \right) \\
   &= E\left(
      \frac{1}{4}
      \begin{pmatrix} v^2+(v^*)^2+2 & v^2-(v^*)^2
            \\ v^2-(v^*)^2 & v^2+(v^*)^2-2 \end{pmatrix}
      \right) \\
   &= \frac{1}{2} \begin{pmatrix} 1 & 0 \\ 0 & -1 \end{pmatrix}.
\end{align*}
We conclude that $u$ is an unbalanced Haar unitary, as required.
\end{ex}

The following example follows a similar pattern.

\begin{ex}[Balanced unitary that is not R-diagonal]
Let $v=1\cdot(1,0)$ and $w=1\cdot(0,1)$ be two commuting Haar unitaries in the
group C$^*$-algebra C$^*(\ZZ \times \ZZ)$ with respect to the canonical trace
$\tau$, which is defined by
\[ \tau\left(\sum_{(n,m)\in\ZZ\times\ZZ}c_{(n,m)}\cdot (n,m)\right)
   = c_{(0,0)}. \]
Note that $v^nw^m = 1\cdot (n,m)$, so
\begin{equation}
\label{eqn:joint haar}
   \tau(v^nw^m) = 0\, \text{for all}\, (n,m)\in\ZZ\times\ZZ\setminus\{(0,0)\}.
\end{equation}
Let $A = M_2(C^*(\ZZ \times \ZZ))$, and identify $B=\CC^2$ with the set of
diagonal matrices whose entries are only supported on the identity element
$(0,0)$ of $\ZZ\times\ZZ$.
By defining the conditional expectation $E:A\to B$ by
\[
   E\left(\begin{pmatrix}
         g_{11} & g_{12} \\ g_{21} & g_{22}
         \end{pmatrix}\right)
      = \begin{pmatrix}
         \tau(g_{11}) & 0 \\ 0 & \tau(g_{22})
        \end{pmatrix},
\]
we have the $B$-valued C$^*$-noncommutative probability space $(A,E)$.

Taking $p$ as in~\eqref{eq:p2x2}, we define the unitary
\[
   u = p \otimes v + (1-p) \otimes w\in A.
\]
To show this is a balanced Haar unitary, we consider an unbalanced $B$-valued
$*$-moment
\[
   x=u^{\varepsilon(1)}b_1\cdots u^{\varepsilon(n-1)}b_{n-1}
      u^{\varepsilon(n)},
\]
where $n\in\NN$, $b_1,\ldots,b_{n-1}\in B$, $\varepsilon\in\{1,*\}^n$, and
$m\coloneqq\#\{j \mid \varepsilon(j)=1\}$ is not equal to $n/2$.
Observe that $x$ is a $2\times 2$ matrix $(a_{ij})$,
where each entry is of the form
\[
   a_{ij} = \sum_{k=0}^m\sum_{\ell=0}^{n-m}
            \lambda_{k,\ell}^{(i,j)}v^kw^{m-k}(v^*)^{\ell}(w^*)^{n-m-\ell}
\]
for some $\lambda_{k,\ell}^{(i,j)}\in\CC$.
Since $m\ne n/2$, we cannot have $k=\ell$ and $m-k = n-m-\ell$ occuring
simultaneously.
We deduce from property~\eqref{eqn:joint haar} that each term
$v^kw^{m-k}(v^*)^{\ell}(w^*)^{n-m-\ell}$ inside the sum has trace zero.
Consequently $\tau(a_{ii})=0$ for each $i$ so that $E(x)=0$.
Therefore $u$ is a balanced unitary.

We show $u$ is not $B$-valued R-diagonal by showing that it violates the moment
condition of \myref{Definition}{defn:R-diagonal}.
More precisely, we aim to prove
\[
  E([u^*b_1u - E(u^*b_1u)]b_2 [ub_3u^* - E(ub_3u^*)]) \ne 0
\]
when $b_1=b_2=b_3=1\oplus0\in B$.
Note that $1\oplus0\in B$ is identified with the matrix unit $e_{11}\in\ M_2(\CC)$.
It is straightforward to compute
\begin{gather*}
u^*b_1u-E(u^*b_1u)=\frac14\begin{pmatrix}vw^*+v^*w & vw^*-v^*w \\ -vw^*+v^*w & -vw^*-v^*w\end{pmatrix}, \\
ub_3u^*-E(ub_3u^*)=\frac14\begin{pmatrix}vw^*+v^*w & -vw^*+v^*w \\ vw^*-v^*w & -vw^*-v^*w\end{pmatrix}, \\
E(\big(u^*b_1u - E(u^*b_1u)\big)b_2\big(ub_3u^* - E(ub_3u^*)\big))=\frac18\begin{pmatrix}1&0\\0&-1\end{pmatrix}.
\end{gather*}

Therefore the balanced unitary $u$ is not $B$-valued $R$-diagonal, as required.
\end{ex}

Before presenting the third example, we show that $R$-diagonal unitaries are precisely those
unitaries that appear as polar parts of $R$-diagonal random variables.
Note that every R-diagonal unitary $u$ trivially appears as the polar part of the
polar decomposition $u1_A$ of iself.
The following proposition, which is a generalization of Proposition 5.2 from
\cite{BD-paper}, implies that every unitary appearing in the polar decomposition
of a $B$-valued R-diagonal must be R-diagonal.

\begin{prop}
\label{prop:polar are r-diagonal}
Let $B$ be a von Neumann algebra and $(A,E)$ be a $B$-valued
W$^*$-noncommuta\-tive probability space.
Suppose $a\in A$ is $B$-valued R-diagonal and $a=v|a|$ is its polar decomposition.
Then the partial isometry $v$ is $B$-valued R-diagonal.
\end{prop}
\begin{proof}
Using the amalgamated free product construction for von Neumann algebras, we
may assume without loss of generality that there is a Haar unitary $u\in A$ that
commutes with $B$ and is $*$-free from $a$ over $B$.
Thus $u$ and $v$ are $*$-free over $B$.
Using \myref{Theorem}{thm:R-diagonal}\ref{part:R-diagonal-condition},
we realize that $a$ and $ua=uv|a|$ have the same $B$-valued $*$-distribution.
Since $(uv)|a|$ is the polar decomposition for $ua$.
it follows that $v$ and $uv$ must have the
same $B$-valued $*$-distribution.
Now we may employ
\myref{Theorem}{thm:R-diagonal}\ref{part:R-diagonal-implication}
to get that $v$ is $B$-valued R-diagonal.
\end{proof}

We'll also need the following lemma for our third example.
\begin{lem}
\label{lem:normalizing}
Let $B$ be a unital C$^*$-algebra with a faithful tracial state $\tau$ and
let $(A,E)$ be a
C$^*$-noncommutative probability space such that $\tau\circ E$ is a trace.
Suppose a unitary $u\in A$ normalizes $B$ with the automorphism $\theta$ on $B$ defined by
$\theta(b)=u^*bu$.
Then
\begin{enumerate}[label=(\alph*)]
\item $x\in\ker E$ implies $u^*xu,\,uxu^*\in\ker E$,
\item $E(u^*xu) = \theta(E(x))$ for all $x\in A$,
\item $E(uxu^*) = \theta^{-1}(E(x))$ for all $x\in A$.
\end{enumerate}
\end{lem}
\begin{proof}
Assuming (a) holds and writing $x\in A$ as $x=\bub{x}+E(x)$, we have
$u^*\bub{x}u,u\bub{x}u^*\in\ker(E)$ so that
\[
   E(u^*xu)=E(u^*\bub{x}u)+E(u^*E(x)u)=\theta(E(x))
\]
and
\[
   E(uxu^*)=E(u\bub{x}u^*)+E(uE(x)u^*)=\theta^{-1}(E(x)).
\]
Hence it suffices to prove (a).

Let $x\in\ker(E)$.
By the traciality of $\tau\circ E$, we have
\begin{multline*}
   \tau(E(u^*xu)^*E(u^*xu))
   = (\tau\circ E)(u^*x^*uE(u^*xu))
   = (\tau\circ E)(x^*uE(u^*xu)u^*) \\
   = (\tau\circ E)(x^*\theta^{-1}(E(u^*xu)))
   = \tau(E(x^*)\theta^{-1}(E(u^*xu)))
   = 0.
\end{multline*}
The faithfulness of $\tau$ gives $E(u^*xu)=0$.
Similarly, we get $E(uxu^*)=0$, which completes the proof.
\end{proof}

Now we can go forward with our final example of the section.
\begin{ex}[R-diagonal unitary that is not normalizing]
\label{ex:RdiagNotNormalizing}
Let $B=\CC^2$.
Consider
the $B$-valued circular element $a$ that does not have a free polar decomposition
in Example 6.8 of~\cite{BD-paper}.
This was realized in a $B$-valued W$^*$-noncommutative probability space, which we will here denote $(A,E)$.
It was shown that the trace $\tau$ on $B$ given by $\tau(\lambda_1,\lambda_2)=\frac{\lambda_1+\lambda_2}2$ satisfies that $\tau\circ E$ is a normal,
faithful, tracial state on $A$.
Taking the polar decomposition $u|a|$ of $a$,
Proposition A.1 of \cite{BD-paper} shows that $u$ is a unitary, and so by
\myref{Proposition}{prop:polar are r-diagonal}, $u$ is $B$-valued R-diagonal unitary.
The cumulants of $a$ are given by
\begin{align*}
   \alpha_{(1,2)}(\lambda_1,\lambda_2)
      &= \left(\frac{\lambda_1}{2}, \frac{\lambda_1}{2} + \lambda_2\right) \\
   \alpha_{(2,1)}(\lambda_1,\lambda_2)
      &= \left(\frac{\lambda_1+\lambda_2}{2}, \lambda_2\right).
\end{align*}

We will show that $u$ fails to normalize $B$.
Suppose, to obtain a contradiction, that $u$ normalizes $B$ with the automorphism $\theta$ on $B$ defined by
$\theta(b)=u^*bu$.
By part~(c) of
\myref{Lemma}{lem:normalizing}, we have
\begin{equation}
\begin{aligned}
\label{eqn:normalizing}
   E(aa^*baa^*)
   &= E(u|a|^2u^*bu|a|^2u^*)
   = \theta^{-1}(E(|a|^2\theta(b)|a|^2)) \\
   &= \theta^{-1}(E(a^*a\theta(b)a^*a))
\end{aligned}
\end{equation}
for all $b\in B$.

With the use of the moment cumulant formula, we verify that the leftmost
expression in~\eqref{eqn:normalizing} is equal to
\begin{equation}
\label{eq:LHSnormalizing}
   E(aa^*baa^*)
      = \alpha_{(1,2)}(1_B)b\alpha_{(1,2)}(1_B) 
         + \alpha_{(1,2)}(\alpha_{(2,1)}(b))
\end{equation}
while the rightmost expression in~\eqref{eqn:normalizing} is
\begin{equation}
\label{eq:RHSnormalizing}
\theta^{-1}(E(a^*a\theta(b)a^*a))=\theta^{-1}\bigg(\alpha_{(2,1)}(1_B)\theta(b)\alpha_{(2,1)}(1_B)+\alpha_{(2,1)}\big(\alpha_{(1,2)}(\theta(b))\big)\bigg).
\end{equation}
Set $b=(1,0)$.
Then we get that the expression~\eqref{eq:LHSnormalizing} equals$(\frac12,\frac14)$.
There are only two automorphisms on $B=\CC^2$ and we can easily compute the value of the expression~\eqref{eq:RHSnormalizing} in both cases.
If $\theta$ is the identity on $B$, then~\eqref{eq:RHSnormalizing} equals $(\frac54,\frac14)$, while
if $\theta$ is the flip on $B$ (i.e.,
$\theta((\lambda_1,\lambda_2))=(\lambda_2,\lambda_1)$), then~\eqref{eq:RHSnormalizing} equals $(\frac94,\frac14)$.
In both cases we have verified that equation~\eqref{eqn:normalizing} fails to hold.
Therefore the R-diagonal unitary $u$ does not normalize $B$, as required.
\end{ex}

\section{Bipolar decompositions}
\label{sec:bipolar}

In this section, we assume that $B$ is a W$^*$-algebra
and that $(A,E)$ is a $B$-valued W$^*$-noncommutative probability space with $E$ faithful.
Since we will be concerned with at most countably many different elements of $A$ at a time, we may without loss of generality assume that $A$ and $B$ have weakly dense
countable subsets, and, thus, that $B$ has a normal, faithful state $\phi_B$.
We may assume that $A$ acts on a Hilbert space via GNS represenation associated to the state $\phi_B\circ E$.

Since $A$ is a W$^*$-algebra, every element $a\in A$ has a unique polar decomposition $a=v|a|$, where $|a|=(a^*a)^{1/2}$ and $v$ is a partial isometry
such that $v^*v$ equals the support projection $1_{(0,\infty)}(a^*a)$ of $a$.

\begin{defn}
A \textit{bipolar decomposition} of a $B$-valued random variable $a\in A$
is a pair $(u,x)$ in some
$B$-valued W$^*$-noncommutative probability space $(A',E')$ with $E'$ faithful,
such that $u$ is a partial
isometry, $x$ is self-adjoint, $u^*ux=x$ and the $B$-valued $*$-distributions of $a$ and $ux$ coincide.
We say the decomposition $(u,x)$ is {\em minimal} if $u^*u$ equals the support projection of $x$, namely, $1_{(0,\infty)}(x^2)$.
We say that two bipolar decompositions $(u,x)$ and $(\ut,\xt)$ are {\em equivalent} if, as pairs of $B$-valued random variables, they have the same
joint $*$-distributions.
\end{defn}

We use the name ``bipolar'' because of the positive and negative directions (poles) on the real axis.
If $(u,x)$ as above is a bipolar decompostion of $a$, then we may write $x=s|x|$ for a symmetry (namely, a self-adjoint unitary) $s\in A'$ that commutes with $x$.
Then, of course, $(us,|x|)$ is also a bipolar decomposition of $a$.

\begin{defn}
We say that a bipolar decomposition $(u,x)$ of $a$ (as above) is
\begin{itemize}
\item[$\bullet$] {\em unitary} if $u$ is a unitary element; 
\item[$\bullet$] {\em tracial} if there is a normal tracial state $\tau_B$ on $B$ so that $\tau_B\circ E'$ is a trace on the $*$-algebra generated by $u$ and $x$;
(note that this implies that the element $a$ is tracial, too);
\item[$\bullet$] {\em standard} if there is a symmetry $s\in A'$ such that $x=s|x|$ and such that $s$ commutes with $x$, with $u$
and with every $b\in B$ (for the copy of $B$ embedded in $A'$), or if $(u,x)$ is equivalent to a bipolar decomposition for which such a symmetry $s$ exists;
\item[$\bullet$] {\em even} if all odd moments of $x$ vanish, namely, if
\[
E'(xb_1xb_2\cdots xb_{2n}x)=0
\]
for every $n\ge1$ and all $b_1,\ldots,b_{2n}\in B$.
\item[$\bullet$] {\em positive} if $x\ge0$; 
\item[$\bullet$] {\em free} if $u$ and $x$ are $*$-free over $B$ (with respect to the conditional expectation $E'$);
\item[$\bullet$] {\em normalizing} if it is unitary and $u$ normalizes the algebra $B$,
namely, if $u^*bu=\theta(b)$ for every $b\in B$, for some automorphism $\theta$ of $B$.
\end{itemize}
\end{defn}

Of course, the polar decomposition $a=v|a|$ of $a\in A$ gives rise to a bipolar decomposition $(v,|a|)$ of $a$ that is minimal and positive.

\begin{lem}
Every $B$-valued random variable $a\in A$ has a bipolar decomposition that is standard, even and minimal.
If $a$ is tracial, then this bipolar decomposition can also be taken to be tracial.
\end{lem}
\begin{proof}
Consider the polar decomposition $a=w|a|$ of $a$.
Let $A'=A\oplus A$ and consider the embedding $i:B\to A'$ given by $i(b)=b\oplus b$.
Using $i$ to identifying $B$ with its image in $A'$, we have the conditional expectation $E':A'\to B$
given by
$E'(a_1\oplus a_2)=\frac12(E(a_1)+E(a_2))$.
In $A'$, consider $u=w\oplus(-w)$ and $x=|a|+(-|a|)$.
Then $(u,x)$ is a bipolar decomposition that is minimal.
We see that it is standard, using the symmetry $s=1\oplus(-1)$.
It is even because for every integer $n\ge1$ and every $b_1,\ldots,b_{n-1}\in B$, we have
\[
E'(x\,i(b_1)\,x\,i(b_2)\cdots x\,i(b_{n-1})\,x)=\frac12(1+(-1)^n)E(|a|\,b_1\,|a|\,b_2\cdots |a|\,b_{n-1}\,|a|)
\]
and this vanishes whenever $n$ is odd.
\end{proof}

We don't quite have the analogue of \myref{Proposition}{prop:polar are r-diagonal} for bipolar decompositions, but the next lemma is used in a similar way.
\begin{lem}\label{lem:takeR-diag}
Suppose a $B$-valued R-diagonal element $a$ has a bipolar decomposition $(v,x)$.
Then $a$ also has a bipolar decomposition $(v',x')$, where $x'$ has the same distribution as $x$ and where $v'$ is $B$-valued R-diagonal.
Furthermore, if $(v,x)$ is tracial, unitary, minimal, standard, free or normalizing, then also
$(v',x')$ can be taken to be tracial, unitary, minimal, standard, free, or normalizing, respectively.
\end{lem}
\begin{proof}
This follows directly from \myref{Theorem}{thm:R-diagonal}.
Indeed, if the bipolar decomposition $(v,x)$ is found in a $B$-valued noncommutative probability space $(A',E')$, then by taking an amalgamated free product (over $B$)
with the W$^*$-algebra
$L^\infty[0,1]\otimes B$ and an appropriate conditional expectation,
we may assume that there is a Haar unitary $w\in A'$ that commutes with $B$ is $*$-free from $\{v,x\}$
with respect to $E'$.
By \myref{Theorem}{thm:R-diagonal}, $wvx$ has the same $*$-distribution as $a$.
Thus, $(wv,x)$ is a bipolar decomposition of $a$.
We let $v'=wv$ and $x'=x$.
Invoking  \myref{Theorem}{thm:R-diagonal} again, we have that $v'$ is $B$-valued R-diagonal.
The properties of being tracial, unitary, minimal, standard, free and normalizing carry over from $(v,x)$ to $(v',x')$.
\end{proof}

It is well known and follows from Proposition~2.6 of~\cite{NSS01} that, in the scalar-valued setting (namely, when $B=\CC$), every tracial R-diagonal element $a$
has the same $*$-distribution as $uh$ in a suitable von Neumann algebra endowed with a normal, faithful, tracial state $\tau$, where $h\ge0$ and where $u$ is
a Haar unitary such that $u$ and $h$ are $*$-free with respect to $\tau$.

An analogous statement does not hold for general $B$.
This follows in the case $B=\CC^2$ by Example~6.8 of~\cite{BD-paper},
which exhibits a tracial circular $\CC^2$-valued random variable
that cannot have a bipolar decomposition that is free.
We will go into more detail about this in \myref{Section}{section:c2}.

The following is part of Proposition~5.5 in \cite{BD-paper}, adapted to the case of W$^*$-noncommutative probability spaces.
The proof of equivalence found in \cite{BD-paper} is easily seen to carry over to this case, by employing the W$^*$-versions of the constructions
of crossed products and amalgamated free products.
\begin{thm}[\cite{BD-paper}]
\label{thm:general-even-decomp}
Let $(A,E)$ be a $B$-valued W$^*$-noncommutative probability space.
Fix a random variable $a\in A$ and an automorphism $\theta$ of $B$.
Then the following are equivalent:
\begin{enumerate}
\item[(i)] $a$ is $B$-valued R-diagonal and for each $k\ge1$, the $k$th order $B$-valued
cumulant maps $\beta_k^{(1)}$ and $\beta_k^{(2)}$ satisfy
\begin{equation}
\begin{aligned}
\label{eqn:R-auto}
   \beta_k^{(2)}(&b_1,\theta(b_2),b_3,\ldots,\theta(b_{2k-2}),b_{2k-1}) \\
   &= \theta(\beta_k^{(1)}(\theta(b_1),b_2,\theta(b_3),\ldots,
         b_{2k-2},\theta(b_{2k-1})))
\end{aligned}
\end{equation}
for all $b_1,\ldots,b_{2k-1}\in B$
\item[(ii)] there exists a bipolar decomposition $(u,x)$ for $a$ 
that is free, even and normalizing, where $u$ is a Haar unitary and
where the implemented automorphism of $B$ is $u^*bu=\theta(b)$.
\item[(iii)] there exists a bipolar decomposition $(u,x)$ for $a$ 
that is free and normalizing, where $u$ is a Haar unitary and
where the implemented automorphism of $B$ is $u^*bu=\theta(b)$.
\end{enumerate}
If the above equivalent conditions hold,
then letting $(A',E')$ be the $B$-valued W$^*$-noncommutative probability space in which the bipolar decomposition $(u,x)$ exists,
the even moments of the self-adjoint element $x$
are given by the formula
\begin{equation}
E'(xb_1xb_2\cdots xb_{2n-1}x)=E(a^*\theta^{-1}(b_1)ab_2a^*\theta^{-1}b_3a\cdots b_{2k-2}a^*\theta^{-1}(b_{2k-1})a),
\end{equation}
which holds for all $n\in\NN$ and all $b_1,\ldots,b_{2n-1}\in B$.
\end{thm}

We do not obtain a free polar decomposition in the above theorem.
The following example demonstrates that a $B$-valued R-diagonal element
satisfying the cumulant condition~\eqref{eqn:R-auto} need not admit a free
polar decomposition with a normalizing unitary.
This is the example from the erratum of~\cite{BD-paper}.

\begin{ex}
\label{ex:even-no-polar}
Let $a$ be a circular element in a scalar-valued W$^*$-non\-com\-mu\-ta\-tive
probability space $(A_0,\tau_0)$ where $\tau_0$ is a normal, faithful, tracial state,
with $\tau_0(a^*a)=1$.
Consider the polar decomposition $a=w|a|$ of $a$.
By Proposition~2.6 from~\cite{V-paper}, $w$ is unitary and, in fact, a Haar unitary.

Let $B$ be a W$^*$-algebra different from $\CC$ with a normal, faithful, tracial
state $\tau_B$.
Consider the von Neumann algebra free product
\[
(A,\tau)=(A_0,\tau_0)*(B,\tau_B)
\]
and let
$E:A\to B$ be the $\tau$-preserving conditional expectation.
Then $(A,E)$ is a $B$-valued W$^*$-noncommutative probability space and
$\tau\circ E = \tau$.
Theorem~12 of \cite{SS-paper} gives a way to express the $B$-valued cumulants
of $a$ in terms of its scalar cumulants and the trace $\tau_B$.
In particular, this implies that $a$ is a $B$-valued circular element whose cumulant
maps $\beta_k^{(1)}$ and $\beta_k^{(2)}$ satisfy
\[
   \beta_1^{(1)}(b)=\beta_1^{(2)}(b) = \tau_B(b)1_B
\]
for all $b\in B$ and $\beta_k^{(1)}=\beta_k^{(2)}=0$ for all $k\ge 2$.
Thus, $a$ satisfies condition~(i) of \myref{Theorem}{thm:general-even-decomp} whenever $\theta$ is $\tau_B$-preserving; i.e.,
$\tau_B\circ\theta=\tau_B$.

Since the polar decomposition of $a$ is $w|a|$ and since $w$ is $*$-free from $B$ with respect to $\tau$, it cannot normalize $B$.
Indeed, let $b\in B$ and $\tau_B(b)=0$.
Then for all $b'\in B$, $\tau(w^*bwb')=0$ by freeness, so $w^*bw\notin B$.
\end{ex}

\begin{rem}\label{rem:utheta}
Without covering all details, 
let us describe an explicit realization of the normalizing free bipolar decomposition of the circular element $a$ from \myref{Example}{ex:even-no-polar}
that is guaranteed to exist
by \myref{Theorem}{thm:general-even-decomp}, for each $\tau_B$-preserving automorphism $\theta$.
Consider the W$^*$-algebra crossed product $B\rtimes_\theta\ZZ$.
Let $E_B:B\rtimes_\theta\ZZ\to B$ be the usual normal conditional expectation and take the trace $\tau_1=\tau_B\circ E_B$.
Let $u$ be the unitary in $B\rtimes_\theta\ZZ$ that implements the automorphism $\theta$.
Let $\tau_2$ be the trace on $L^\infty([-2,2])$ given by integration with respect to Lebesgue measure and let $x\in L^\infty([-2,2])$
be a self-adjoint element that is semicircular with respect to $\tau_2$.
Take the W$^*$-algebra free product
\[
(A,\tau)=(B\rtimes_\theta\ZZ,\tau_1)*(L^\infty([-2,2]),\tau_2)
\]
and consider $c=ux\in A$.
Let $E:A\to B$ be the $\tau$-preserving conditional expectation.
One easily sees that $u$ and $x$ are $*$-free over $B$ with respect to $E$.
Thus, $(u,x)$ is a free normalizing bipolar decomposition for $c$, with $u$ implementing the automorphism $\theta$ and with $x$ even.
The polar decomposition of $c$ is $(us)|x|$, where $s\in L^\infty[-2,2]$ is a symmetry with $\tau_2(s)=0$.
In order to show that $c$ has the same $B$-valued $*$-distribution as the circular operator $a$ described at the begining of this example,
we must see that, with respect to the trace $\tau$,
\begin{enumerate}[label=(\roman*)]
\item\label{part:usHaar} $us$ is a Haar unitary
\item\label{part:us|x|free} $us$ and $|x|$ are $*$-free
\item\label{part:cBfree} $c$ is $*$-free from $B$.
\end{enumerate}
Part~\ref{part:usHaar} is immediate, by freeness of $u$ and $s$, since both have trace zero.
Part~\ref{part:us|x|free} follows in straightforward manner, using that $u$ and $\{s,|x|\}$ are free and that $s$ is a symmetry that commutes with $|x|$.
Since we already know that $|x|$ and $B$ are free and have proved part~\ref{part:us|x|free},
in order to prove part~\ref{part:cBfree}, it will suffice to show that
$us$ and $B$ are $*$-free.
This follows by considering words and using that $s$ is free from $B$, that $uB\subseteq\ker\tau_1$
and that conjugation with $u$ implements a $\tau_B$-preserving automorphism of $B$.
\end{rem}

Bipolar decompositions that are even and normalizing are not unique in general.
This is clear upon revisiting \myref{Example}{ex:even-no-polar}, which is an
example of a $B$-valued circular $a$ whose cumulants satisfy~\eqref{eqn:R-auto}
for every trace preserving automorphism $\theta$.
Therefore \myref{Theorem}{thm:general-even-decomp} says that, for each of these
automorphisms $\theta$, there exists a bipolar decomposition
$(u_\theta,x_\theta)$ of $a$ that is free, even and normalizing, where $u_\theta$ is a Haar unitary and
$u_\theta^*bu_\theta^{}=\theta(b)$ for all $b\in B$.
By taking different $\theta$, one obtains different bipolar decompositions of $a$ that are even and normalizing.
(See \myref{Remark}{rem:utheta}.)
This can happen even when $B=\CC^2$.

The next result shows that under certain natural nondegeneracy conditions,
if one free bipolar decomposition is normalizing, then all other free bipolar decompositions
are normalizing and implement the same automorphism.
\begin{thm}
\label{thm:degen}
Let $a$ be a $B$-valued random variable in a $B$-valued W$^*$-non\-com\-mu\-ta\-tive
probability space $(A,E)$
and assume that either of the subspaces
\begin{equation}
\label{eqn:degen}
\lspan\{E((a^*a)^k) \mid k\ge0\}
\text{ or }
\lspan\{E((aa^*)^k) \mid k\ge0\}
\end{equation}
is weakly dense in $B$.
Suppose that $a$ has free bipolar decompositions $(u,x)$ and $(\ut,\xt)$ in $B$-valued W$^*$-noncommutative probability spaces $(A',E')$ and $(\At,\Et)$,
respectively, with $\Et$ faithful.
Suppose $u$ and $\ut$ are unitaries satisfying $E'(u)=0=\Et(\ut)$ and suppose that $u$ normalizes $B$.
Then $\ut$ normalizes $B$, and induces the same automorphism, namely, $\ut^*b\ut=u^*bu$ for all $b\in B$.
\end{thm}
\begin{proof}
Let $\theta$ be the automorphism of $B$ induced by $u$, namely, $u^*bu=\theta(b)$.
We have
\begin{multline*}
E((aa^*)^k)= E'((uxxu^*)^k)=E'(ux^{2k}u^*)=E'(u(E'(x^{2k}))u^*) \\
=\theta^{-1}(E'(x^{2k}))=\theta^{-1}(E'(((xu^*)(ux))^k)=\theta^{-1}(E((a^*a)^k),
\end{multline*}
where for the third equality we used $*$-freeness of $u$ and $x$ as well as $E'(u)=0$. 
Since $\theta$ is an automorphism, density of either of the subspaces in~\eqref{eqn:degen} implies density of the other.
We also get
\begin{multline*}
\Et(\ut(E((a^*a)^k)\ut^*)=\Et(\ut(\Et(x^{2k}))\ut^*)=\Et(\ut\xt^{2k}\ut^*) \\
=E((aa^*)^k)=\theta^{-1}(E((a^*a)^k)),
\end{multline*}
where for the third equality we used $*$-freeness of $\ut$ and $\xt$ as well as $\Et(\ut)=0$.
By the density of the subspaces in~\eqref{eqn:degen}, we have
\[
\Et(\ut b \ut^*)=\theta^{-1}(b)
\]
for all $b\in B$.
Thus,
\begin{align*}
\Et((\ut b&\ut^*-\theta^{-1}(b))^*(\ut b\ut^*-\theta^{-1}(b))) \\
&=\Et(\ut b^*b\ut^*)-\theta^{-1}(b^*)\Et(\ut b\ut^*)-\Et(\ut b^*\ut^*)\theta^{-1}(b)+\theta^{-1}(b^*)\theta^{-1}(b) \\
&=\Et(\ut b^*b\ut^*)-\theta^{-1}(b^*b)=0.
\end{align*}
By faithfulness of $\Et$, we conclude $\ut b\ut^*=\theta^{-1}(b)$ for all $b\in B$.
\end{proof}

\section{$\CC^2$-valued circular elements}
\label{section:c2}

Recall that a $B$-valued random variable $a$ is called
\textit{$B$-valued circular} if $\alpha_j=0$ whenever
$j\in J\setminus\{(1,2),(2,1)\}$, where $J=\bigcup_{n\ge 1}\{1,2\}^n$, and
$(\alpha_j)_{j\in J}$ are the cumulant maps associated to the pair
$(a_1,a_2)=(a,a^*)$.

Proposition 5.1 from \cite{BD-paper} implies that if
$a$ is a $B$-valued circular element with associated cumulant maps
$\alpha_{(1,2)}$ and $\alpha_{(2,1)}$ and if $\tau$ is a trace on $B$,
then $\tau\circ E$ is a trace on $\Alg(B\cup\{a,a^*\})$ if and only if
\begin{equation}
\label{eqn:tracial}
   \tau(\alpha_{(1,2)}(b_1)b_2) = \tau(b_1\alpha_{(2,1)}(b_2))
   \qquad \text{for all $b_1,b_2\in B$}.
\end{equation}

Our main result in this section, for the case $B=\CC^2$,
concerns $\CC^2$-valued circular elements with the tracial property~\eqref{eqn:tracial}.
\begin{thm}
\label{thm:main}
Let $a$ be a $\CC^2$-valued circular random variable that is tracial and has a free bipolar decomposition $(u,x)$, where
$u$ is unitary.
Then
\begin{equation}
\label{eqn:original-auto}
\alpha_{(2,1)} = \theta \circ \alpha_{(1,2)} \circ \theta
\end{equation}
for some automorphism $\theta$ on $\CC^2$.
\end{thm}

Combining this with \myref{Theorem}{thm:general-even-decomp} and \myref{Theorem}{thm:degen}, we have:

\begin{cor}
\label{cor:C2-free-even-decomp}
Let $a$ be a $\CC^2$-valued circular random variable in a $\CC^2$-valued
W$^*$-noncommutative probability space $(A,E)$ and assume that $a$ is tracial.
Then $a$ has a free bipolar decomposition $(u,x)$ with $u$ unitary if and only if
it has a normalizing bipolar decomposition that is also free and even.
Moreover, suppose
\[
\lspan\{E((a^*a)^k) \mid k\ge0\}=\CC^2
\text{ or }
\lspan\{E((aa^*)^k) \mid k\ge0\}=\CC^2.
\]
Then for every free bipolar decomposition $(u,x)$ of $a$ in a
$\CC^2$-valued W$^*$-noncommutative probability space $(A',E')$, where $E'$ is
faithful and $u$ is unitary, the unitary $u$ normalizes $\CC^2$ (namely, the copy of $\CC^2$ that is the image of $E'$).
\end{cor}

Our proof of \myref{Theorem}{thm:main} relies on intricate calculations
for $\CC^2$-valued random variables.
However, we first prove some more general results that we will using during these calculations, namely,\myref{Lemmas}{lem:1}, \ref{lem:g1g2}, \ref{lem:m0} and~\ref{lem:GH}, below.

\begin{lem}
\label{lem:1}
Let $B$ be a unital $*$-algebra and $(A,E)$ a $B$-valued $*$-non\-com\-mu\-ta\-tive probability space.
Suppose $y=ux$ for $u,x\in A$ with $u$ a unitary satisfying $E(u)=0$, with $x=x^*$
and with $u$ and $x$ $*$-free.
Then for all integers
$n,m,k\ge 0$, we have
\begin{equation*}
   E((yy^*)^n)=E(u E(x^{2n})u^*)=E(u E((y^*y)^n)u^*)
\end{equation*}
and
\begin{equation}
\label{eq:Exes}
\begin{aligned}
   E((yy^*)^n(y^*y)^m(yy^*)^k)
      &= E(uE(x^{2n}E(u^*E(x^{2m})u)x^{2k})u^*) \\
      &\quad -E(uE(x^{2n})E(u^*E(x^{2m})u)E(x^{2k})u^*) \\
      &\quad +E(uE(x^{2n})u^*E(x^{2m})uE(x^{2k})u^*)
\end{aligned}
\end{equation}
\end{lem}
\begin{proof}
We have $yy^*=ux^2u^*$ so that
\[
    E((yy^*)^n)
   =  E(ux^{2n}u^*)
   =  E(u E(x^{2n})u^*)
   =  E(u E((y^*y)^n)u^*),
\]
where the second equality is due to freeness.

We use the notation, for a $B$-valued random variable $y$, 
\[ \bub{y} = y^\circ \coloneqq y-E(y), \]
so that $y=\bub{y}+E(y)$ and $E(\bub{y})=0$.
Now
\begin{align}
\label{eqn:2.1}
    E(&(yy^*)^n(y^*y)^m(yy^*)^k) \nonumber \\
   &=  E(ux^{2n}u^*x^{2m}ux^{2k}u^*) \nonumber \\
   &=  E(ux^{2n}u^*(x^{2m})^\circ ux^{2k}u^*)
    +  E(ux^{2n}u^* E(x^{2m})ux^{2k}u^*).
\end{align}
We consider each of the terms in the last expression separately.
The first one is zero because
\begin{align*}
    E(u&x^{2n}u^*(x^{2m})^\circ ux^{2k}u^*) \\
   &=  E(u (x^{2n})^\circ u^*(x^{2m})^\circ u (x^{2k})^\circ u^*)
    +  E(u (x^{2n})^\circ u^*(x^{2m})^\circ u  E(x^{2k}) u^*) \\
   &\quad +  E(u  E(x^{2n}) u^*(x^{2m})^\circ u (x^{2k})^\circ u^*)
          +  E(u  E(x^{2n}) u^*(x^{2m})^\circ u  E(x^{2k}) u^*),
\end{align*}
and $*$-freeness of $u$ and $x$ ensures that each of these terms vanishes.
We rewrite the second term on the right-hand side of~\eqref{eqn:2.1} to obtain
\begin{align}
\label{eqn:2.2}
   E(u&x^{2n}u^* E(x^{2m})ux^{2k}u^*) \nonumber \\
   &=  E(ux^{2n}[u^* E(x^{2m})u]^\circ x^{2k}u^*)
   +  E(ux^{2n} E[u^* E(x^{2m})u] x^{2k}u^*).
\end{align}
The first term from~\eqref{eqn:2.2} expands to
\begin{align*}
    &E(ux^{2n}[u^* E(x^{2m})u]^\circ x^{2k}u^*) \\
   &=  E(u (x^{2n})^\circ [u^* E(x^{2m})u]^\circ (x^{2k})^\circ u^*) 
    +  E(u (x^{2n})^\circ [u^* E(x^{2m})u]^\circ  E(x^{2k}) u^*) \\
   &\; +  E(u  E(x^{2n}) [u^* E(x^{2m})u]^\circ (x^{2k})^\circ u^*)
   +  E(u  E(x^{2n}) [u^* E(x^{2m})u]^\circ  E(x^{2k}) u^*),
\end{align*}
and $*$-freeness of $u$ and $x$ implies that only the fourth of these terms may be nonzero.
Hence the first term on the right-hand side of~\eqref{eqn:2.2} satisfies
\begin{align*}
    &E(ux^{2n}[u^* E(x^{2m})u]^\circ x^{2k}u^*)
   =  E(u  E(x^{2n}) [u^* E(x^{2m})u]^\circ  E(x^{2k}) u^*) \\
   &=  E(u  E(x^{2n}) u^* E(x^{2m})u  E(x^{2k}) u^*) 
   -E(u  E(x^{2n}) E(u^* E(x^{2m})u)  E(x^{2k}) u^*).
\end{align*}
The other term from~\eqref{eqn:2.2} can be rewritten as
\begin{align*}
    E(ux^{2n} &E[u^* E(x^{2m})u] x^{2k}u^*) \\
   &=  E(u [x^{2n} E[u^* E(x^{2m})u] x^{2k}]^\circ u^*)
    +  E(u  E[x^{2n} E[u^* E(x^{2m})u] x^{2k}] u^*),
\end{align*}
and freeness forces the first of these two terms to be zero.
This shows that the second term on the right-hand side of~\eqref{eqn:2.2}
satisfies
\begin{align*}
    E(ux^{2n} E[u^* E(x^{2m})u] x^{2k}u^*)
   =  E(u  E(x^{2n} E(u^* E(x^{2m})u) x^{2k}) u^*).
\end{align*}
Putting everything together, we get~\eqref{eq:Exes}.
\end{proof}

\begin{lem}
\label{lem:g1g2}
Let $B$ be a unital $*$-algebra and $(A,E)$ a $B$-valued $*$-non\-com\-mu\-ta\-tive probability space.
Assume that $y$ is a $B$-valued circular random variable in $(A,E)$,
with cumulants $\beta_{(1,2)}$
and $\beta_{(2,1)}$ associated to $(y_1,y_2)=(y,y^*)$.
Let
\begin{equation}
\label{eq:g1g2def}
g_1(n)=E((yy^*)^n),\qquad
g_2(n)=E((y^*y)^n).
\end{equation}
Then the maps $g_1$ and
$g_2$ can be computed recursively by $g_1(0)=g_2(0)=1_{\CC^2}$, and, for all $n\ge1$,
\begin{align}
   g_1(n)
   &= \sum_{i=1}^n\beta_{(1,2)}(g_2(i-1))g_1(n-i), \label{eq:g1n} \\
   g_2(n)
   &= \sum_{i=1}^n\beta_{(2,1)}(g_1(i-1))g_2(n-i). \label{eq:g2n}
\end{align}
Furthermore, for every $n,m,k\ge0$, we have
\begin{equation}
\label{eq:Enmk}
E((yy^*)^n(y^*y)^m(yy^*)^k)= E((yy^*)^n E((y^*y)^m)(yy^*)^k).
\end{equation}
\end{lem}
\begin{proof}
Clearly $g_1(0)=g_2(0)=1$.
The recursive relations are easy consequences of the moment-cumulant formula.
Since $y$ is circular, the only noncrossing partitions that need to be
considered are noncrossing pair partitions for which each set in the partition
pairs an $y$ with an $y^*$.
By considering the $n$ possibilities for the position of the $y^*$ in the word $(yy^*)^n$ that is
paired with the left-most $y$ in a given noncrossing pair partition, we obtain
\begin{align*}
   g_1(n)
   &=  E((yy^*)^n)
   = \sum_{i=1}^n\beta_{(1,2)}( E((y^*y)^{i-1})) E((yy^*)^{n-i}) \\
   &= \sum_{i=1}^n\beta_{(1,2)}(g_2(i-1))g_1(n-i).
\end{align*}
Similarly, by considering the position of the $y$ that is paired with the left-most $y^*$, we get
\begin{align*}
   g_2(n)
   &=  E((y^*y)^n)
   = \sum_{i=1}^n\beta_{(2,1)}( E((yy^*)^{i-1})) E((y^*y)^{n-i}) \\
   &= \sum_{i=1}^n\beta_{(2,1)}(g_1(i-1))g_2(n-i).
\end{align*}
Finally, the equality~\eqref{eq:Enmk} can be proved using the moment-cumulant formula.
Indeed, any noncrossing pairing of the $2(n+m+k)$ gotten from writing out $(yy^*)^n(y^*y)^m(yy^*)^k$
that pairs only instances of $y$ with instances of $y^*$, must pair all those from the middle block of $2m$ terms 
with others from that same block.
\end{proof}

\begin{lem}
\label{lem:m0}
Let $B$ be a unital $*$-algebra and let $(A,E)$ be a $B$-valued $*$-noncommutative probability space.
Let $y$ be a $B$-valued circular random variable in $(A,E)$ that can be written $y=ux$ for $u,x\in A$ as in \myref{Lemma}{lem:1}
and let $g_1$ and $g_2$ be as defined in~\eqref{eq:g1g2def} of \myref{Lemma}{lem:g1g2}.
Consider the linear map $M:B^{\otimes 3}\to B$ defined by
\[ M(b_1\otimes b_2\otimes b_3) =  E(ub_1u^*b_2ub_3u^*) \]
and the map $M_0:\NN_0^3\to B$ given by
\begin{equation}
\label{eqn:m0m}
M_0(n,m,k) =  M(g_2(n)\otimes g_2(m)\otimes g_2(k)).
\end{equation}
Then
\begin{equation}
\begin{aligned}
\label{eqn:m0}
   M_0(n,m,k)
   &= G(n, g_2(m), k)
   -N_1(H(n, N_2(g_2(m)), k)) \\
   &\quad+N_1(g_2(n)N_2(g_2(m))g_2(k)),
\end{aligned}
\end{equation}
where
\begin{align*}
   H  &:\NN_0\times B\times\NN_0\to B;
         &&(n,b,k) \mapsto  E((y^*y)^nb(y^*y)^k) \\
   G  &:\NN_0\times B\times\NN_0\to B;
         &&(n,b,k) \mapsto  E((yy^*)^nb(yy^*)^k) \\
   N_1&:B\to B; &&b \mapsto  E(ubu^*) \\
   N_2&:B\to B; &&b \mapsto  E(u^*bu).
\end{align*}
\end{lem}
\begin{proof}
The relations proved in \myref{Lemma}{lem:1} become
\begin{equation}
\label{eqn:lem1,1}
   g_1(n) =  E(ug_2(n)u^*).
\end{equation}
and
\begin{align*}
    E((yy^*)^n(y^*y)^m(yy^*)^k) 
   &=  E(u E(x^{2n} E(u^*g_2(m)u)x^{2k})u^*) \\ 
   &\quad-  E(ug_2(n) E(u^*g_2(m)u)g_2(k)u^*) \\
   &\quad+  E(ug_2(n)u^*g_2(m)ug_2(k)u^*).
\end{align*}
This yields
\begin{align*}
   M_0(n,m,k) 
   &=  E(ug_2(n)u^*g_2(m)ug_2(k)u^*) \\
   &=  E((yy^*)^n(y^*y)^m(yy^*)^k) - E(u E(x^{2n} E(u^*g_2(m)u)x^{2k})u^*) \\
      &\quad + E(ug_2(n) E(u^*g_2(m)u)g_2(k)u^*) \\
   &=  E((yy^*)^n E((y^*y)^m)(yy^*)^k) -N_1(H(n, N_2(g_2(m)), k)) \\
      &\quad +N_1(g_2(n)N_2(g_2(m))g_2(k)) \\
   &= G(n, g_2(m), k)
      -N_1(H(n, N_2(g_2(m)), k))  \\
      &\quad+N_1(g_2(n)N_2(g_2(m))g_2(k)),
\end{align*}
where in the third equality we used~\eqref{eq:Enmk} from \myref{Lemma}{lem:g1g2}
to rewrite the term $E((yy^*)^n(y^*y)^m(yy^*)^k)$.
\end{proof}

\begin{lem}
\label{lem:GH}
In the setting of \myref{Lemma}{lem:m0},
let $\beta_{(1,2)}$ and $\beta_{(2,1)}$ be the cumulant maps associated to $(y_1,y_2)=(y,y^*)$.
In addition to the maps $G$ and $H$, consider also the maps
\begin{align*}
G'&:\NN_0\times B\times\NN_0\to B;
  &&(n,b,k) \mapsto E((y^*y)^ny^*by(y^*y)^k) \\
H'&:\NN_0\times B\times\NN_0\to B;
  &&(n,b,k) \mapsto E((yy^*)^nyby^*(yy^*)^k).
\end{align*}
Then $G$ and $G'$ can be computed recursively by
\begin{alignat*}{2}
   G(0,b,k) &= bg_1(k), \\
   G'(0,b,k) &= \sum_{j=0}^k\beta_{(2,1)}(bg_1(j))g_2(k-j), \displaybreak[1] \\
   G(n,b,k) &= \sum_{i=1}^n\beta_{(1,2)}(g_2(i-1))G(n-i,b,k) \\
            &\quad+ \sum_{j=1}^k\beta_{(1,2)}(G'(n-1,b,j-1))g_1(k-j),  \qquad&(n\ge1) \displaybreak[1] \\ 
   G'(n,b,k) &= \sum_{i=1}^n\beta_{(2,1)}(g_1(i-1)) G'(n-i,b,k) \\
            &\quad+ \sum_{j=0}^k\beta_{(2,1)}(G(n,b,j))g_2(k-j), &(n\ge1).
\end{alignat*}
Moreover, $H$ and $H'$ can be computed recursively by
\begin{alignat*}{2}
   H(0,b,k) &= bg_2(k), \\
   H'(0,b,k) &= \sum_{j=0}^k\beta_{(1,2)}(bg_2(j))g_1(k-j), \displaybreak[1] \\
   H(n,b,k) &= \sum_{i=1}^n\beta_{(2,1)}(g_1(i-1))H(n-i,b,k) \\
            &\quad+ \sum_{j=1}^k\beta_{(2,1)}(H'(n-1,b,j-1))g_2(k-j), \qquad&(n\ge1) \displaybreak[1] \\ 
   H'(n,b,k) &= \sum_{i=1}^n\beta_{(1,2)}(g_2(i-1)) H'(n-i,b,k) \\
            &\quad+ \sum_{j=0}^k\beta_{(1,2)}(H(n,b,j))g_1(k-j),  &(n\ge1).
\end{alignat*}
\end{lem}
\begin{proof}
The moment-cumulant formula yields
\begin{align*}
   H(n,b,k)
   &=  E((y^*y)^nb(y^*y)^k) \\
   &= \sum_{i=1}^n\beta_{(2,1)}(g_1(i-1)) E((y^*y)^{n-i}b(y^*y)^k) \\
   &\quad+ \sum_{j=1}^k\beta_{(2,1)}( E((yy^*)^{n-1}yby^*(yy^*)^{j-1}))g_2(k-j) \displaybreak[1] \\
   &= \sum_{i=1}^n\beta_{(2,1)}(g_1(i-1))H(n-i,b,k) \\
    &\quad+ \sum_{j=1}^k\beta_{(2,1)}(H'(n-1,b,j-1))g_2(k-j),
\end{align*}
and another application of the moment-cumulant formula gives
\begin{align*}
   H'(&n,b,k)
   =  E((yy^*)^nyby^*(yy^*)^k) \\
   &= \sum_{i=1}^n\beta_{(1,2)}(g_2(i-1))  E((yy^*)^{n-i}yby^*(yy^*)^k) \\
    &\quad+ \beta_{(1,2)}( E((y^*y)^nb))g_1(k) 
   +\sum_{j=1}^k\beta_{(1,2)}( E((y^*y)^nb(y^*y)^j))g_1(k-j) \\
   &= \sum_{i=1}^n\beta_{(1,2)}(g_2(i-1))  E((yy^*)^{n-i}yby^*(yy^*)^k) \\
   &\quad+ \sum_{j=0}^k\beta_{(1,2)}( E((y^*y)^nb(y^*y)^j))g_1(k-j) \\
   &= \sum_{i=1}^n\beta_{(1,2)}(g_2(i-1)) H'(n-i,b,k)
   + \sum_{j=0}^k\beta_{(1,2)}(H(n,b,j))g_1(k-j).
\end{align*}
Together, these yield the recursive formulas for $H$ and $H'$.

Similarly, two applications of the moment-cumulant formula yield
the recursive formulas for $G$ and $G'$.
\end{proof}

\begin{proof}[Proof of \myref{Theorem}{thm:main}]
Let $(A,E)$ be a $\CC^2$-valued $*$-noncommutative probability space, and fix
a $\CC^2$-valued circular element $a\in A$ with associated cumulant maps $\alpha_{(1,2)}$ and $\alpha_{(2,1)}$ for the pair $(a,a^*)$.
We assume that there is a faithful trace $\tau$ on $\CC^2$, given by $\tau(x,y)=qx+(1-q)y$, where $q\in(0,1)$, the corresponding
functional $\tau\circ E$ is a trace on $A$.
We assume, moreover, that the circular element $a$ has
a free bipolar decomposition $(u,x)$, with $u$ unitary.
By \myref{Lemma}{lem:takeR-diag} and \myref{Theorem}{thm:unitaries},
we may assume that the unitary $u$ is a Haar unitary.

We will denote elements of $\CC^2$ as ordered pairs of complex numbers.
Then $\alpha_{(1,2)}$ and $\alpha_{(2,1)}$ are given by
\begin{equation}
\label{eq:alph1221}
\begin{aligned}
   \alpha_{(1,2)}(x,y)
   &= (r_{11}x + r_{12}y, r_{21}x + r_{22}y) \\
   \alpha_{(2,1)}(x,y)
   &= (s_{11}x + s_{12}y, s_{21}x + s_{22}y),
\end{aligned}
\end{equation}
for some parameters $r_{ij}$, $s_{ij}$.
By positivity of $\alpha_{(1,2)}$ and $\alpha_{(2,1)}$ we have $r_{ij}\ge0$ and $s_{ij}\ge0$ for each $i,j\in\{1,2\}$,
and each such choice of $r_{ij}$ and $s_{ij}$ yields completely positive cumulant maps $\alpha_{(1,2)}$ and $\alpha_{(2,1)}$
given by~\eqref{eq:alph1221}.

Since $\tau\circ E$ is a trace, for all
$b_1=(x_1,y_1),b_2=(x_2,y_2)\in \CC^2$, we have
\[ 
   \tau(\alpha_{(1,2)}(b_1)b_2) = \tau(b_1\alpha_{(2,1)}(b_2)),
\]
which is equivalent to
\begin{align*}
&qr_{11}x_1x_2+qr_{12}y_1x_2+(1-q)r_{21}x_1y_2+(1-q)r_{22}y_1y_2 \\
   &= qs_{11}x_1x_2+qs_{12}x_1y_2+(1-q)s_{21}y_1x_2+(1-q)s_{22}y_1y_2.
 \end{align*}
   This holds for every $b_1,b_2\in\CC^2$ if and only if
\begin{equation}
\label{eq:selim}
r_{11}=s_{11},\quad r_{22}=s_{22}, \quad qr_{12}=(1-q)s_{21},\quad (1-q)r_{21}=qs_{12}.
\end{equation}
Using~\eqref{eq:selim} to eliminate the $s_{ij}$, the cumulants are given by
\begin{equation}
\label{eqn:cumulants}
\begin{aligned}
   \alpha_{(1,2)}(x,y)
   &= (r_{11}x + r_{12}y, r_{21}x + r_{22}y) \\
   \alpha_{(2,1)}(x,y)
   &= (r_{11}x + \dfrac{1-q}{q}r_{21}y, \dfrac{q}{1-q}r_{12}x + r_{22}y).
\end{aligned}
\end{equation}

There are only two automorphisms on $\CC^2$, so it's easy using~\eqref{eqn:cumulants} to determine what
the automorphism condition~\eqref{eqn:original-auto} means in terms of the
parameters $q,r_{11},r_{12},r_{21},\allowbreak r_{22}$.
If $\theta$ is the identity automorphism, then~\eqref{eqn:original-auto} is equivalent to
$qr_{12}=(1-q)r_{21}$.
If $\theta$ is the flip automorphism defined by
$\theta((x,y))=(y,x)$, then~\eqref{eqn:original-auto} is equivalent to
\begin{align*}
   (r_{11}x+\dfrac{1-q}{q}r_{21}y,\dfrac{q}{1-q}r_{12}x+r_{22}y)
   = (r_{21}y+r_{22}x,r_{11}y+r_{12}x),
\end{align*}
for all $x,y\in\CC$,
which holds if and only if $r_{11}=r_{22}$ and $q=1/2$.
Hence there exists some automorphism $\theta$ on $\CC^2$ satisfying~\eqref{eqn:original-auto}
if and only if the parameters satisfy
\begin{equation}
\label{eqn:new-auto}
   qr_{12} = (1-q)r_{21}
   \quad\OR\quad (r_{11} = r_{22} \quad\AND\quad q = 1/2).
\end{equation}

Our strategy to prove \myref{Theorem}{thm:main} is to 
suppose that~\eqref{eqn:new-auto} fails to hold,
namely, to assume
\begin{align}
\label{eqn:auto}
   qr_{12} \ne (1-q)r_{21}
   \quad\AND\quad ( r_{11}\ne r_{22} \quad\OR\quad q\ne 1/2 ),
\end{align}
and to obtain a contradiction.

We will be applying \myref{Lemmas}{lem:1}, \ref{lem:g1g2}, \ref{lem:m0} and~\ref{lem:GH} in the case $y=a$ and $\beta_{(i,j)}=\alpha_{(i,j)}$.
Thus, with the notation used in \myref{Lemma}{lem:g1g2}, we have the functions $g_1,g_2:\NN_0\to \CC^2$ given by
\begin{align*}
   g_1(n) &=  E((aa^*)^n)=E(ux^{2n}u^*) \\
   g_2(n) &=  E((a^*a)^n)=E(x^{2n}).
\end{align*}

We now use~\eqref{eqn:lem1,1} to show that
$1_{\CC^2}=(1,1)=g_2(0)\in\CC^2$ and $g_2(1)= E(a^*a)$ are linearly independent.
Indeed, if $\alpha_{(2,1)}(1_{\CC^2})=g_2(1)=(c,c)$ for some scalar $c$, then we
have by~\eqref{eqn:lem1,1} that
\begin{multline*}
\alpha_{(1,2)}(1_{\CC^2}) = E(aa^*)
   =g_1(1)
   = E(ug_2(1)u^*) \\
   =  E(u(c,c)u^*)=(c,c)=\alpha_{(2,1)}(1_{\CC^2}).
\end{multline*}
Comparing the first components of $\alpha_{(1,2)}(1_{\CC^2})$ and $\alpha_{(2,1)}(1_{\CC^2})$
from~\eqref{eqn:cumulants} yields the equation
$r_{11}+r_{12}=r_{11}+\frac{1-q}{q} r_{21}$, and thus
$qr_{12}=(1-q)r_{21}$.
This contradicts~\eqref{eqn:auto}.
Therefore $g_2(0)=1_{\CC^2}$ and $g_2(1)$ are linearly independent, so we may
write every $b\in\CC^2$ as a linear combination of $g_2(0)$ and $g_2(1)$.

Since it will be important to work with these linear combinations, we 
describe the coordinate functionals $P_1,P_2:\CC^2\to\CC$ for the basis
$\{g_2(0),g_2(1)\}$ of $\CC^2$; namely, the linear functionals $P_1,P_2$
satisfying
\begin{equation}
\label{eq:P1P2coefs}
b=P_1(b) g_2(0)+ P_2(b) g_2(1), \quad b\in\CC^2.
\end{equation}
Taking $b=(x,y)$ and writing $g_2(1)=(g_2(1)_1,g_2(1)_2)$, we solve the
linear system
\begin{align*}
   x &= P_1((x,y))+P_2((x,y))g_2(1)_1 \\
   y &= P_1((x,y))+P_2((x,y))g_2(1)_2
\end{align*}
to obtain
\begin{equation}
\label{eq:P1P2}
\begin{aligned}
P_1((x,y)) &= x-P_2((x,y))g_2(1)_1, \\
P_2((x,y)) &= (x-y)/(g_2(1)_1-g_2(1)_2).
\end{aligned}
\end{equation}

From~\eqref{eqn:lem1,1}, if
$0=\sum_{k=0}^nc_kg_2(k)$ for some scalars $c_k$, then
\begin{align*}
   \sum_{k=0}^nc_kg_1(k)
   = \sum_{k=0}^nc_k E(ug_2(k)u^*)
   =  E\left(u\sum_{k=0}^kc_kg_2(k)u^*\right)
   =  E(u0u^*)
   = 0.
\end{align*}
That is, any vanishing linear combination of $g_2(k)$'s yields the same
vanishing linear combination of $g_1(k)$'s.
In particular, from
\[
g_2(n)=P_1(g_2(n)) g_2(0) + P_2(g_2(n)) g_2(1)
\]
we get
\begin{equation}
\label{eqn:lin}
   g_1(n)=P_1(g_2(n)) g_1(0) + P_2(g_2(n)) g_1(1),
   \qquad \text{for all}\ n\in\NN_0.
\end{equation}

Equations~\eqref{eq:g1n} and~\eqref{eq:g2n} from \myref{Lemma}{lem:g1g2} (in the case $y=a$ and $\beta_{(i,j)}=\alpha_{(i,j)}$)
and also the formulas~\eqref{eqn:cumulants} for the cumulants $\alpha_{(i,j)}$ allow us to express, for each $n$,
the components of $g_1(n)$ and $g_2(n)$ as rational functions
in the parameters $q,r_{11},r_{12},r_{21},r_{22}$.
We then use \eqref{eqn:lin} in the cases $n=2$ and $n=3$ to obtain some rather long algebraic equations involving these parameters.
From here, we apply Mathematica's \cite{Mat} Reduce function to these equations.
Incoporating also the assumptions from~\eqref{eqn:auto}, we conclude that the parameters satisfy one of the following
conditions:
\begin{enumerate}[leftmargin=2cm,label=Case \Roman*:,ref=\Roman*]
\item $r_{11}=r_{21}$ and $r_{12} \ne r_{22}$ and $q=r_{11}/(r_{11}+r_{22})$
      \label{case:1}
\item $r_{11}\ne r_{21}$ and $r_{12} = r_{22}$ and $q=r_{11}/(r_{11}+r_{22})$
      \label{case:2}
\item $r_{11} = r_{21}$ and $r_{12} = r_{22}$.  \label{case:3}
\end{enumerate}
The details of the calculations that were carried out in Mathematica
can be found in the Mathematica Notebook~\cite{DG}.

We will use the function $M$ and, more specifically, $M_0$, defined in \myrefsplit{Lemma}{lem:m0} (with $y=a$ and $\beta_{(i,j)}=\alpha_{(i,j)}$).
Using~\eqref{eqn:m0}, $M_0$ is defined in terms of functions $G$, $H$, $N_1$ and $N_2$, as well as the quantities $g_1$ and $g_2$.
As we have already seen, $g_1$ and $g_2$ are expressed in terms of the parameters $r_{ij}$ and $q$ by \myref{Lemma}{lem:g1g2}, while using the recursive formulas given in \myref{Lemma}{lem:GH},
$G$ and $H$ can also be expressed in terms of these parameters.
Now to focus in $N_1$,
using~\eqref{eq:P1P2coefs} and~\eqref{eqn:lem1,1}, we get
\begin{align*}
   N_1(b)
   = P_1(b) E(ug_2(0)u^*)+P_2(b) E(ug_2(1)u^*)
   = P_1(b)g_1(0) + P_2(b)g_1(1).
\end{align*}
Thus, with help of~\eqref{eq:P1P2}, the map $N_1$ can be expressed in terms of the parameters $r_{ij}$ and $q$.

Regarding $N_2$, we use the $*$-freeness of $u$ and $x$ over $\CC^2$ to
obtain, for all $b\in\CC^2$,
\begin{align}
   \label{eqn:m1m3}
   \alpha_{(1,2)}(N_2(b)) 
   &=  E(a E(u^*bu)a^*) \nonumber
   =  E(ux E(u^*bu)xu^*) \\ 
   &=  E(u E(x E(u^*bu)x)u^*) 
   =  E(u E(xu^*bux)u^*) \\ \nonumber
   &=  E(u E(a^*ba)u^*) 
   = N_1(\alpha_{(2,1)}(b)), 
\end{align}
which yields the relation $\alpha_{(1,2)}\circ N_2 = N_1\circ \alpha_{(2,1)}$.
In \myref{Cases}{case:1} and~\ref{case:2}, since $0<q<1$, both $r_{11}$ and $r_{22}$
must be nonzero.
Hence, $r_{11}r_{22}\ne r_{12}r_{21}$ and, consequently, the cumulant map $\alpha_{(1,2)}$
is invertible.
Thus, \eqref{eqn:m1m3} implies
\[
N_2=\alpha_{(1,2)}^{-1}\circ N_1 \circ \alpha_{(2,1)}.
\]
Consequently, $N_2$ and also $M_0$ can also be expressed in terms of the parameters $r_{ij}$ and $q$, which we write for $m,n,k\in\NN_0$, as
\begin{equation}
\label{eq:M0params}
M_0(m,n,k)=F(m,n,k;r_{11},r_{r12},r_{21},r_{22},q).
\end{equation}

In \myref{Case}{case:3},  $\alpha_{(1,2)}$ is
not invertible and we cannot write $N_2$ in terms of $N_1$.
Instead, since $N_2$ is linear and positive, we write
\begin{align*}
   N_2((x,y)) &= (m_{11}x+m_{12}y,m_{21}x+m_{22}y)
\end{align*}
for parameters
$m_{11},m_{12},m_{21},m_{22}\ge 0$.
Since
\[ 1_{\CC^2}=N_2(1_{\CC^2})=(m_{11}+m_{12},m_{21}+m_{22}) \]
we get $m_{12}=1-m_{11}$ and $m_{21}=1-m_{22}$, which we incorporate into the description of $N_2$.
In \myref{Case}{case:3}, we express $M_0$ in terms of
$q,r_{11},r_{12},r_{21},r_{22}$ together with the new parameters
$m_{11},m_{22}$,
 which we write for $m,n,k\in\NN_0$, as
\begin{equation}
\label{eq:M0params3}
M_0(m,n,k)=F_3(m,n,k;r_{11},r_{r12},r_{21},r_{22},q,m_{11},m_{22}).
\end{equation}

In fact, since $g_2(0)$ and $g_2(1)$ form a basis for $\CC^2$, $M$ 
is determined by the values of $M_0$ on the eight triples in $\{0,1\}^3$.
More specifically, fixing $(n,m,k)\in\NN_0^3$,
taking $b_1=g_2(n)$, $b_2=g_2(m)$ and $b_3=g_2(k)$ and
writing $c_i:=P_1(b_i)$ and $d_i:=P_2(b_i)$ for $i=1,2,3$,
we have
\begin{align}
   &M_0(n,m,k)=M(b_1\otimes b_2\otimes b_3) \notag \\ \notag
   &\;= M((c_1g_2(0)+d_1g_2(1))\otimes(c_2g_2(0)+d_2g_2(1))\otimes(c_3g_2(0)+d_3g_2(1))) \\ \notag
   &\;= c_1c_2c_3M_0(0,0,0)
    + c_1c_2d_3M_0(0,0,1)
    + c_1d_2c_3M_0(0,1,0) \label{eqn:m-by-mo} \\ 
   &\quad +d_1c_2c_3M_0(1,0,0)
    + c_1d_2d_3M_0(0,1,1)
    + d_1c_2d_3M_0(1,0,1) \\ \notag
   &\quad +d_1d_2c_3M_0(1,1,0)
    +d_1d_2d_3M_0(1,1,1).
\end{align}
We will use this together with the expressions~\eqref{eq:M0params} in \myref{Cases}{case:1} and~\ref{case:2} and~\eqref{eq:M0params3} in \myref{Case}{case:3}
to get algebraic relations involving the parameters, which will lead to contradictions.
Once again, we will use Mathematica's \cite{Mat} Reduce function.
The details of the calculations that were carried out in Mathematica
can be found in the Mathematica Notebook~\cite{DG}.

Suppose we are in \myref{Case}{case:1}:
$r_{11}=r_{21}$ and $r_{12}\ne r_{22}$ and $q=r_{11}/(r_{11}+r_{22})$.
Since $0<q<1$ and either $q\ne 1/2$ or $r_{11}\ne r_{22}$, we must have
$r_{11}\ne 0$ and $r_{22}\ne r_{11}$.
We may without loss of generality  assume $r_{11}=1$, for replacing $a$ with $ta$ for any real number $t>0$ amounts to replacing $\alpha_{(1,2)}$ with $t^2\alpha_{(1,2)}$ and
and $\alpha_{(2,1)}$ with $t^2\alpha_{(2,1)}$.
With the assumptions $r_{11}=r_{21}=1$, $r_{12}\ne r_{22}$,
$q=r_{11}/(r_{11}+r_{22})$, and $r_{22}\ne 1$,
Mathematica's reduce function says that~\eqref{eqn:m-by-mo} with
$(n,m,k)=(3,1,3)$ implies that $r_{22}$ is negative or non-real.
This contradicts that $r_{22}$ must be positive.

Now suppose we are in \myref{Case}{case:2}:
$r_{11}\ne r_{21}$ and $r_{12}=r_{22}$ and $q=r_{11}/(r_{11}+r_{22})$.
Similarly to in \myref{Case}{case:1}, we may
assume $r_{22}=1$ and $r_{11}\ne r_{22}$.
By testing~\eqref{eqn:m-by-mo} again with $(n,m,k)=(3,1,3)$, we deduce
from the output of Mathematica's reduce function that $r_{11}$ is negative or non-real.
This contradicts that $r_{11}$ must be positive.

Finally, we consider \myref{Case}{case:3}:
$r_{11}=r_{21}$ and $r_{12}=r_{22}$, which
we split into four disjoint subcases:
\begin{enumerate}[leftmargin=2.5cm,label=Subcase III.\arabic*:,ref=III.\arabic*]
\item $r_{11}=r_{21}=r_{12}=r_{22}$
      \label{subcase:1}
\item $0=r_{11}=r_{21}\ne r_{12}=r_{22}$
      \label{subcase:2}
\item $r_{11}=r_{21}\ne r_{12}=r_{22}=0$
      \label{subcase:2'}
\item $0\ne r_{11}=r_{21}\ne r_{12}=r_{22}\ne0$.
      \label{subcase:3}
\end{enumerate}
We will use~\eqref{eqn:m-by-mo} multiple times with different triples $(n,m,k)$, now using the expression~\eqref{eq:M0params3}, of course.
Using again using Mathematica's
Reduce function, in each subcase we obtain a contradiction.

In \myref{Subcase}{subcase:1},
we cannot have all parameters $r_{ij}$ equal to $0$, and we cannot have $q=1/2$, for
either of these would contradict~\eqref{eqn:auto}.
After rescaling $a$, if necessary, we may assume $r_{11}=r_{12}=r_{21}=r_{22}=1$ and $q\ne 1/2$.
Using~\eqref{eqn:m-by-mo} with $(n,m,k)=(2,1,1)$
yields $m_{22}=1-m_{11}$.
After adding this assumption, and now using~\eqref{eqn:m-by-mo} with 
$(n,m,k)=(2,1,3)$, we obtain that either $q$ is non-real, negative, greater
than $1$, or equal to $1-m_{22}$.
Since only the latter is possible, we incorporate this identity and use~\eqref{eqn:m-by-mo}
with $(n,m,k)=(3,1,3)$.
Mathematica's Reduce function concludes that $q$ must be outside the interval $(0,1)$,
which is a contradiction.

In \myref{Subcase}{subcase:2}, after rescaling, we may assume $r_{12}=r_{22}=1$.
Using~\eqref{eqn:m-by-mo} with $(n,m,k)=(2,1,1)$, we find that $m_{22}=1-q$.
After adding this assumption, we use~\eqref{eqn:m-by-mo} using $(n,m,k)=(3,1,1)$ and
get $q=0$, which is a contradiction.

In \myref{Subcase}{subcase:2'}, after rescaling, we may assume $r_{11}=r_{21}=1$.
Using~\eqref{eqn:m-by-mo} with $(n,m,k)=(2,1,1)$, we conclude $m_{11}=q$.
After adding this assumption, we use~\eqref{eqn:m-by-mo} with $(n,m,k)=(3,1,1)$ 
and we get $q=1$, which is a contradiction.

Lastly, we investigate \myref{Subcase}{subcase:3}.
Without loss of generality, we set $r_{11}=r_{21}=1$.
Now the condition $qr_{12}\ne(1-q)r_{21}$ from~\eqref{eqn:auto} implies
$q\ne 1/(1+r_{22})$.
We use~\eqref{eqn:m-by-mo} with
$(n,m,k)=(2,1,1)$, yielding
\[
m_{11}=q+r_{22}^2-m_{22}r_{22}^2-qr_{22}^2.
\]
With this constraint, using~\eqref{eqn:m-by-mo} with $(n,m,k)=(1,1,3)$ yields
$m_{11}=q$ and $m_{22}=1-q$.
Adding in these equalities and using~\eqref{eqn:m-by-mo} with $(n,m,k)=(3,1,3)$
yields a list of possible values for $q$, all of which are outside the interval $(0,1)$.
This gives us our final contradiction.
\end{proof}

\section{Open Questions}
\label{sec:qns}

We finish with two open questions.
As before, assume $B$ is a W$^*$-algebra and $(A,E)$ is a $B$-valued W$^*$-noncommutative probability space.

\begin{ques}\label{qu:free}
Suppose a tracial $B$-valued circular element $a\in A$ has a tracial, free bipolar decomposition $(u,x)$ that is minimal and suppose $u$ is unitary.
Must $a$ also have a free polar decomposition, namely, if $a=v|a|$ is the polar decomposition
of $a$, must $v$ and $|a|$ be $*$-free with respect to $E$?
\end{ques}

Note that \myref{Example}{ex:even-no-polar} does not provide a counter-example.
In that case, from the polar decomposition $a=w|a|$, the elements $w$ and $|a|$ are $*$-free over $B$
with respect to the conditional expectation $E$.
It is just that $w$ does not normalize $B$.
Note also that, if the answer to \myref{Question}{qu:free} is ``yes,'' then under a nondegeneracy condition
(which fails in the case of \myref{Example}{ex:even-no-polar})
if $u$ normalizes $B$, then
the unitary $v$ must also normalize $B$, by
\myref{Theorem}{thm:degen}.

\begin{ques}\label{qu:normalize}
Suppose a tracial $B$-valued circular element $a\in A$  has a tracial, free bipolar decomposition $(u,x)$ with $u$ unitary.
Must it also have a free bipolar decomposition $(u',x')$ where $u'$ is a unitary that normalizes $B$?
\end{ques}

Of course, \myref{Theorem}{thm:main} answers this question affirmatively in the case $B=\CC^2$.
What about more generally?

\section{Declaration}

The authors have no relevant financial or non-financial interests to disclose.

\end{document}